\documentclass[10pt]{article}%
\usepackage[english]{babel}
\usepackage{amsmath}
\usepackage{amsthm}
\usepackage{amsfonts}
\usepackage{amssymb}
\usepackage{graphicx}
\usepackage[all]{xy}
\usepackage{lscape}
\usepackage{flafter}
\usepackage{multirow}%
\newtheorem{theorem}{Theorem}[subsection]

\newtheorem{definition}[theorem]{Definition}
\newtheorem{example}[theorem]{Example}
\newtheorem{lemma}[theorem]{Lemma}
\newtheorem{proposition}[theorem]{Proposition}

\newtheorem{notation}[theorem]{Notation}
\begin{document}

\title{Projective limits of local shift morphisms}
\author{
Patrick Cabau 
}
\maketitle

\begin{abstract}
We define the notion of projective limit of local shift morphisms of type
$\left(  r,s\right)  $ and endow the space of such mathematical objects with
an adapted differential structure. The notion of shift Poisson tensor $P$ on a
Hilbert tower corresponds to such morphisms which are antisymmetric and whose
Schouten bracket $\left[  P,P\right]  $ vanishes. We illustrate this notion
with the example of the famous KdV equation on the circle $\mathbb{S}^{1}$ for
which one can associate a couple of compatible Poisson tensors of this type on
the Hilbert tower $\left(  H^{n}(\mathbb{S}^{1})\right)  _{n\in\mathbb{N}%
^{\ast}}$.

\end{abstract}



\smallskip
\thanks{
\small{The author is grateful to Professor Fernand Pelletier for helpful comments.}
} 

\section{Introduction \label{_Introduction}}

This article is dealing with projective limits of shift Poisson tensors on
Hilbert towers whose set can be endowed with a Fr\'{e}chet structure.\newline

The paper is organized as follows. Section
\ref{_ProjectiveLimitsOfBanachSpacesAndDifferentiability} introduces the basic
notions and results on projective limits of Banach spaces and an adapted
notion of differentiability on such spaces. In section \ref{_ShiftOperators},
we introduce shift operators on direct limits of Banach spaces and endow the
set of projective limits of such operators with a Fr\'{e}chet structure
(Theorem \ref{T_FrechetStructureOnProjectiveLimitOfLnrs}). In section
\ref{_ProjectiveSequenceOfLocalShiftMorphisms}, we introduce the notion of
local shift morphism and study the smoothness of projective limits of such
operators (Theorem
\ref{T_SmoothStructureOnProjectiveLimitOfLocalShiftMorphims}). In section
\ref{_HilbertTowers}, we consider the particular case of Hilbert towers that
appears as an adapted framework to describe some PDEs. Section
\ref{_ShiftHilbertPoissonTensors} is devoted to the notion of shift Hilbert
Poisson tensors $P$, corresponding to a projective limit of antisymmetric
local shift morphisms defined on a Hilbert tower where the Schouten bracket
$[P,P]$ vanishes. As a fundamental example, we consider the KdV equation on
the circle $\mathbb{S}^{1}$ (cf. \cite{KapMak}) for which there exists a
couple of compatible shift Hilbert Poisson tensors on the projective limit of the Sobolev spaces $H^{n}\left(  \mathbb{S}^{1}\right)  $.

\section{Projective limits of Banach spaces and differentiability\label{_ProjectiveLimitsOfBanachSpacesAndDifferentiability}}

In a lot of situations in global analysis and Mathematical Physics, the
framework of Banach or Hilbert spaces is not adapted any more. In some cases,
the projective limits of such spaces must be adopted. For such Fr\'{e}chet
spaces, the differentiation method proposed by J.A. Leslie fits well to the
requirements of this geometrical situation.\newline We can remark that the
convenient setting, defined by A. Fr\"{o}licher and A. Kriegl (see
\cite{FroKri} and \cite{KriMic}), could have been used. This framework is
adapted to various structures (e.g. for convenient partial Poisson structures
as defined in \cite{Pel}).

\subsection{Projective limits of topological spaces\label{__ProjectiveLimitsTopologicalSpaces}}

\begin{definition}
$\left\{  \left(  X_{i},\delta_{i}^{j}\right)  \right\}  _{\left(  i,j\right)
\in\mathbb{N}^{2},\ j\geq i}$ is a \textit{projective sequence of topological
spaces} if

\begin{description}
\item[\textbf{(PSTS 1)}] for all $i\in\mathbb{N},$ $X_{i}$ is a topological space;

\item[\textbf{(PSTS 2)}] for all $i,$ $j$ $\in\mathbb{N},$ such that $j\geq
i,$ $\delta_{i}^{j}:X_{j}\rightarrow X_{i}$ is a continuous mapping;

\item[\textbf{(PSTS 3)}] for all $i\in\mathbb{N},$ $\delta_{i}^{i}={Id}%
_{X_{i}}$;

\item[\textbf{(PSTS 4)}] for all integers $i\leq j\leq k$, $\delta_{i}%
^{j}\circ\delta_{j}^{k}=\delta_{i}^{k}$.
\end{description}
\end{definition}

\begin{definition}
An element $\left(  x_{i}\right)  _{i\in\mathbb{N}}$ of the product
$\prod\limits_{i\in N}X_{i}$ is called a \textit{thread} if for all $j\geq i$,
$\delta_{i}^{j}\left(  x_{j}\right)  =x_{i}$.\newline The set
$X=\underleftarrow{\lim}X_{i}$ of such elements, endowed with the finest
topology for which all the projections $\delta_{i}:X\rightarrow X_{i}$ are
continuous, is called \textit{projective limit of the sequence} $\left\{
\left(  X_{i},\delta_{i}^{j}\right)  \right\}  _{\left(  i,j\right)
\in\mathbb{N}^{2},\ j\geq i}$.
\end{definition}

A basis of the topology of $X$ is constituted by the subsets $\left(
\delta_{i}\right)  ^{-1}\left(  U_{i}\right)  $ where $U_{i}$ is an open
subset of $X_{i}$ (and so $\delta_{i}^{j}$ is open).

\begin{definition}
Let $\left\{  \left(  X_{i},\delta_{i}^{j}\right)  \right\}  _{\left(
i,j\right)  \in\mathbb{N}^{2},\ j\geq i}$ and $\left\{  \left(  Y_{i}%
,\gamma_{i}^{j}\right)  \right\}  _{\left(  i,j\right)  \in\mathbb{N}%
^{2},\ j\geq i}$be two projective systems whose respective projective limits
are $X$ and $Y$.\newline A sequence $\left(  f_{i}\right)  _{i\in\mathbb{N}}$
of continuous mappings $f_{i}:X_{i}\rightarrow Y_{i}$, satisfying for all
$i,j\in\mathbb{N},$ $j\geq i,$ the condition%
\[
\gamma_{i}^{j}\circ f_{j}=f_{i}\circ\delta_{i}^{j}%
\]
is called a \textit{projective system of mappings}.
\end{definition}

The projective limit of this sequence is the mapping%
\[%
\begin{array}
[c]{cccc}%
f: & X & \rightarrow & Y\\
& \left(  x_{i}\right)  _{i\in\mathbb{N}} & \mapsto & \left(  f_{i}\left(
x_{i}\right)  \right)  _{i\in\mathbb{N}}%
\end{array}
\]

The mapping $f$ is continuous and is a homeomorphism if all the $f_{i}$ are
homeomorphisms (cf. \cite{AbbMan}).

\subsection{Differentiability \label{__Differentiability}}

We first introduce the notion of differentiability \textit{\`{a} la Leslie}
between Hausdorff locally convex vector spaces $E$ and $F$ which corresponds
to a particular case of the G\^{a}teaux derivative. For full details, the
reader is referred to \cite{Les} and \cite{DoGaVa}. Unlike the classical
framework of Banach spaces, the derivative does not involve the space of
continuous linear maps $\mathcal{L}\left(  E,F\right)  $ which has no
reasonable structure.

\begin{definition}
\label{D_DifferentiableMapBetweenHLCVS} Let $E$ and $F$ be two Hausdorff
locally convex vector spaces and let $U$ be an open subset of $E$. A
continuous map $f:U\longrightarrow F$ is said to be differentiable at $x\in U$
if there exits a continuous linear map $Df_{x}:E\longrightarrow F$ such that%
\[
R\left(  t,v\right)  =\left\vert
\begin{array}
[c]{cc}%
\dfrac{f\left(  x+tv\right)  -f\left(  x\right)  -Df_{x}\left(  tv\right)
}{t}, & t\neq0\\
0, & t=0
\end{array}
\right.
\]
is continuous at every $\left(  0,v\right)  \in\mathbb{R}\times F$. The map
$Df_{x}$ is called the derivative (or differential) of $f$ at $x$.

The map is said to be differentiable if it is differentiable at every $x\in U$.
\end{definition}

Note that, in this case, $Df_{x}$ is uniquely determined.

\begin{definition}
A continuous map $f:U\longrightarrow F$ from an open subset $U$ of a Hausdorff
locally convex vector space $E$ to a space of the same type $F$ is called
C$^{1}$-differentiable if it is differentiable at every $x\in U$, and if the
derivative%
\[%
\begin{array}
[c]{cccc}%
Df: & U\times E & \longrightarrow & F\\
& \left(  x,v\right)  & \mapsto & Df_{x}\left(  v\right)
\end{array}
\]
is continuous.
\end{definition}

The notion of $C^{n}$-differentiability ($n\geq2$) can be defined by induction
(cf. \cite{DoGaVa}, Definition 2.2.3) and allows to define the $C^{\infty}%
$-differentiability \textit{\`{a} la Leslie} which corresponds to the
$C^{\infty}$-differentiability in the ordinary case.

We then have the following properties:

\begin{description}
\item[\textbf{(PDL 1)}] Every continuous linear map $f:E\longrightarrow F$ is
Leslie $C^{\infty}$ and $Df=F$;

\item[\textbf{(PDL 2)}] The differential at $x$ satisfies the relation%
\[
Df_{x}\left(  h\right)  =\underset{t\longrightarrow0}{\lim}\dfrac{f\left(
x+th\right)  -f\left(  x\right)  }{t}%
\]

\item[(PDL 3)] The chain rules holds.
\end{description}

\subsection{Differentiability on projective limits\label{__DifferentiabilityOnProjectiveLimits}}

The connection between projective limits of maps and differentiation is given
by the following result (\cite{DoGaVa}, Propositions 2.3.11 and 2.3.12).

\begin{proposition}
\label{P_DiffentiabilityOfProjectiveLimitsOfDifferentialFunctions} Let
$\mathbb{F}_{1}=\underleftarrow{\lim}\mathbb{E}_{1}^{i}$ and $\mathbb{F}%
_{2}=\underleftarrow{\lim}\mathbb{E}_{2}^{i}$ projective limits of Banach
spaces. Let also $f^{i}:U^{i}\longrightarrow\mathbb{E}_{2}^{i}$ be where, for
all $i\in\mathbb{N}$, $U^{i}$ is an open set of $\mathbb{E}_{1}^{i}$. We
assume that $U=\underleftarrow{\lim}U^{i}$ exists and is a non empty open
subset of $\mathbb{F}_{1}$; we also assume that $f=\underleftarrow{\lim}%
f^{i}:U\longrightarrow\mathbb{F}_{2}$ exists. Then we have:

If each $f^{i}$ is differentiable (resp. smooth), then so is $f$ and%
\[
\forall x=\left(  x^{i}\right)  \in U,\ Df_{x}=\underleftarrow{\lim}Df_{x^{i}%
}\text{.}%
\]
\end{proposition}

\section{Shift operators\label{_ShiftOperators}}

In Analysis and Mathematical Physics, Banach representations break down. By
weakening the topological requirement, replacing the norm by a sequence of
semi-norms, one gets the notion of Fr\'{e}chet space. For the subsections
\ref{__FrechetSpaces} (resp. \ref{__TheFrechetSpaceHF1F2}), the reader is
referred to \cite{Bour},  \cite{RobRob} and \cite{Tre} (resp. \cite{DoGaVa}).

\subsection{Fr\'{e}chet spaces\label{__FrechetSpaces}}

\begin{definition}
\label{D_FrechetSpace}A Fr\'{e}chet space is a Hausdorff, locally convex
topological vector space that is metrizable and complete.
\end{definition}

The topology of a Fr\'{e}chet space $\mathbb{F}$ can be induced by a sequence
of semi-norms $\left(  \nu_{n}\right)  _{n\in\mathbb{N}}$ that is complete
with respect to such a sequence.

Recall that $\mathbb{F}$ is complete with respect to this topology if and only
if every sequence $\left(  x_{i}\right)  _{i\in\mathbb{N}}$ in $\mathbb{F}$ is
such that
\[
\forall n\in\mathbb{N},\forall\varepsilon>0,\exists i_{\varepsilon}%
\in\mathbb{N}:\forall\left(  j,k\right)  \in\mathbb{N}^{2},k\geq j\geq
i_{\varepsilon},\nu_{n}\left(  x_{k}-x_{j}\right)  <\varepsilon
\]
converges in $\mathbb{F}$ where the convergence in this Fr\'{e}chet space is
controlled by all the semi-norms $\nu_{n}$:%
\[
\underset{i\longrightarrow+\infty}{\lim}x_{i}=x\quad\Longleftrightarrow
\quad\forall n\in\mathbb{N},\underset{i\longrightarrow+\infty}{\lim}\nu
_{n}\left(  x_{i}-x\right)  =0
\]

\begin{example}
\label{Ex_EspaceDesSuitesReelles}The space of real sequences $\mathbb{R}%
^{\mathbb{N}}=\prod\limits_{n\in N}\mathbb{R}^{n}$ endowed with the usual
topology is a Fr\'{e}chet space where the corresponding sequence of semi-norms
is given by%
\[
\nu_{n}\left(  \left(  x_{i}\right)  _{i\in\mathbb{N}}\right)  =\sum
\limits_{k=0}^{n}\left\vert x_{k}\right\vert
\]

Metrizability is defined from $d$ as follows%
\[
d\left(  x,y\right)  =\sum\limits_{k=0}^{+\infty}\dfrac{\left\vert y_{k}%
-x_{k}\right\vert }{2^{k}\left(  1+\left\vert y_{k}-x_{k}\right\vert \right)
}%
\]

and the completeness is inherited from that of each $\mathbb{R}$ of the
infinite product.
\end{example}

The notion of Fr\'{e}chet space is closely related with the projective limit
of Banach spaces.

If $\left\{  \left(  \mathbb{B}_{n},\left\Vert \ \right\Vert _{n}\right)
\right\}  _{n\in\mathbb{N}}$ is a projective sequence of Banach spaces, then
$\underleftarrow{\lim}\mathbb{B}_{n}$ is a Fr\'{e}chet space (cf.
\cite{DoGaVa}, Theorem 2.3.7) where the sequence $\left(  \nu_{n}\right)
_{n\in\mathbb{N}}$ of semi-norms is given by%
\[
\forall x=\left(  x_{n}\right)  _{n\in\mathbb{N}}\in\underleftarrow{\lim
}\mathbb{B}_{n},\ \nu_{n}\left(  x\right)  =\sum\limits_{i=0}^{n}\left\Vert
x_{n}\right\Vert _{n}%
\]

Conversely, if $\mathbb{F}$ is a Fr\'{e}chet space with associated semi-norms
$\nu_{n}$, the completion $\mathbb{F}_{n}$ of the normed space $\mathbb{F}%
/\ker\nu_{n}$ is a Banach space called the \textit{local Banach space
associated to the semi-norm} $\nu_{n}$. It will be denoted by $\left(
\mathbb{F}_{n},\left\Vert \ \right\Vert _{n}\right)  $ where $\left\Vert
\ \right\Vert _{n}$ is the norm associated to $\nu_{n}$. We then get a
projective system $\left\{  \left(  \mathbb{F}_{i},\pi_{i}^{j}\right)
\right\}  _{\left(  i,j\right)  \in\mathbb{N}^{2},\ j\geq i}$ of Banach spaces
whose bonding maps are%
\[%
\begin{array}
[c]{cccc}%
\pi_{i}^{j}: & \mathbb{F}_{j} & \longrightarrow & \mathbb{F}_{i}\\
& \left[  x+\ker\nu_{j}\right]  _{j} & \longmapsto & \left[  x+\ker\nu
_{i}\right]  _{i}%
\end{array}
\]
where the bracket $\left[  \ \ \right]  _{n}$ corresponds to the associated
equivalence class. $\mathbb{F}$ will be identified with the projective limit
$\underleftarrow{\lim}\mathbb{F}_{i}$ (cf. \cite{DoGaVa}, Theorem 2.3.8).

The representation of Fr\'{e}chet spaces as projective limits of Banach spaces
is very interesting: Issues arising in the Fr\'{e}chet framework can be solved
by considering their components in the Banach factors of the associated
projective sequence. So different pathological entities in the Fr\'{e}chet
framework can be replaced by approximations compatible with the inverse
limits, e.g. ILB-Lie groups (\cite{Omo}) or projective limits of Banach Lie
groups (\cite{Gal1}), manifolds (\cite{AbbMan}), bundles (\cite{Gal2},
\cite{AghSur}), algebroids (\cite{Cab}), connections and differential
equations (\cite{ADGS}).

\subsection{The Fr\'{e}chet space $\mathcal{H}\left(  \mathbb{F}_{1},\mathbb{F}_{2}\right) $ \label{__TheFrechetSpaceHF1F2}}

Let $\mathbb{F}_{1}$ (resp. $\mathbb{F}_{2}$)$\mathbb{\ }$be a Fr\'{e}chet
space and $\left(  \nu_{1}^{n}\right)  _{n\in\mathbb{N}}$ (resp. $\left(
\nu_{2}^{n}\right)  _{n\in\mathbb{N}}$) the sequence of semi-norms of
$\mathbb{F}_{1}$(resp. $\mathbb{F}_{2}$).

Recall (\cite{Vog}, 2.) that a linear map $L:\mathbb{F}_{1}\longrightarrow
\mathbb{F}_{2}$ is \textit{continuous} if%
\[
\forall n\in\mathbb{N},\exists k_{n}\in\mathbb{N},\exists C_{n}>0:\forall
x\in\mathbb{F}_{1},\nu_{2}^{n}\left(  L.x\right)  \leq C_{n}\nu_{1}^{k_{n}%
}\left(  x\right)
\]

The space $\mathcal{L}\left(  \mathbb{F}_{1},\mathbb{F}_{2}\right)  $ of
continuous linear maps between both these Fr\'{e}chet spaces generally drops
out of the Fr\'{e}chet category. Indeed, $\mathcal{L}\left(  \mathbb{F}%
_{1},\mathbb{F}_{2}\right)  $ is a Hausdorff locally convex topological vector
space whose topology is defined by the family of semi-norms $\left\{
p_{n,B}\right\}  $:%
\[
p_{n,B}\left(  L\right)  =\sup\left\{  \nu_{2}^{n}\left(  L.x\right)  ,x\in
B\right\}
\]

where $n\in\mathbb{N}$ and $B$ is any bounded subset of $\mathbb{F}_{1}$
containing $0_{\mathbb{F}_{1}.}$ This topology is not metrizable since the
family $\left\{  p_{n,B}\right\}  $ is not countable. \newline So
$\mathcal{L}\left(  \mathbb{F}_{1},\mathbb{F}_{2}\right)  $ will be replaced,
under certain assumptions, by a projective limit of appropriate functional
spaces as introduced in \cite{Gal2}.

If we denote by $\mathcal{L}\left(  \mathbb{B}_{1}^{n},\mathbb{B}_{2}%
^{n}\right)  $ the space of linear continuous maps (or equivalently bounded
linear maps because $\mathbb{B}_{1}^{n}$ and $\mathbb{B}_{2}^{n}$ are normed
spaces), we then have the following result (\cite{DoGaVa}, Theorem 2.3.10).

\begin{theorem}
\label{T_HF1F2} The space of all continuous linear maps between $\mathbb{F}%
_{1}$ and $\mathbb{F}_{2}$ that can be represented as projective limits%
\[
\mathcal{H}\left(  \mathbb{F}_{1},\mathbb{F}_{2}\right)  =\left\{  \left(
L_{n}\right)  \in\prod\limits_{n\in\mathbb{N}}\mathcal{L}\left(
\mathbb{B}_{1}^{n},\mathbb{B}_{2}^{n}\right)  :\underleftarrow{\lim}%
L_{n}\text{ exists}\right\}
\]
is a Fr\'{e}chet space.
\end{theorem}

For this sequence $\left(  L_{n}\right)  $ of linear maps, for any integer
$0\leq i\leq j$, the following diagram is commutative%
\[%
\begin{array}
[c]{ccc}%
\mathbb{B}_{1}^{i} & \overset{{\vphantom{A}}_{1}\delta_{i}^{j}}{\longleftarrow
} & \mathbb{B}_{1}^{j}\\
L_{i}\downarrow &  & \downarrow L_{j}\\
\mathbb{B}_{2}^{i} & \overset{{\vphantom{A}}_{2}\delta_{i}^{j}}{\longleftarrow
} & \mathbb{B}_{2}^{j}%
\end{array}
\]

\subsection{Shift operators\label{__ShiftOperators}}

We assume that $\mathbb{F}_{1}=\underleftarrow{\lim}\mathbb{B}_{1}^{n}$ (resp.
$\mathbb{F}_{2}=\underleftarrow{\lim}\mathbb{B}_{2}^{n}$) is a Fr\'{e}chet
space where $\left\{  \left(  \mathbb{B}_{1}^{i},_{{\vphantom{A}}1}\delta
_{i}^{j}\right)  ,\left\Vert \ \right\Vert _{1}^{i}\right\}  _{\left(
i,j\right)  \in\mathbb{N}^{2},\ j\geq i}$(resp. $\left\{  \left(
\mathbb{B}_{2}^{i},_{{\vphantom{A}}2}\delta_{i}^{j}\right)  ,\left\Vert
\ \right\Vert _{2}^{i}\right\}  _{\left(  i,j\right)  \in\mathbb{N}%
^{2},\ j\geq i}$) is a projective sequence of Banach spaces.

\begin{definition}
\label{D_ShiftOperator}A linear map $L:\mathbb{B}_{1}^{n+r}\longrightarrow
\mathbb{B}_{2}^{n-s}$ is called a shift operator of base $n$ and type $\left(
r,s\right)  \in\mathbb{N}\times\mathbb{N}$ where $n\geq s$, if there exists
$C_{n}>0$ such that:
\[
\forall x\in\mathbb{B}_{1}^{n+r},\ \left\Vert L.x\right\Vert _{2}^{n-s}\leq
C_{n}\left\Vert x\right\Vert _{1}^{n+r}%
\]
\newline
\end{definition}

\begin{notation}
$\mathcal{L}_{n}^{r,s}\left(  \mathbb{F}_{1},\mathbb{F}_{2}\right)  $ denotes
the set of shift operators of base $n$ and type $\left(  r,s\right)  $.
\end{notation}

\begin{lemma}
$\mathcal{L}_{n}^{r,s}\left(  \mathbb{F}_{1},\mathbb{F}_{2}\right)  $ endowed
with the norm $\left\Vert \ \right\Vert _{L_{n}^{r,s}}$ defined by%
\[
\left\Vert L\right\Vert _{L_{n}^{r,s}}=\sup_{\left\Vert x\right\Vert
_{1}^{n+r}}\left\Vert L.x\right\Vert _{2}^{n-s}%
\]
is a Banach space.
\end{lemma}

A linear operator of base $n$ and type $\left(  r,s\right)  $ is continuous.

\begin{example}
\label{Ex_DifferentialOperatorOrderr}(\cite{Ham}, 1.1.2, Examples (4) and
1.2.3 Examples (3)). Let $X$ be a compact manifold. Then $C^{\infty}\left(
X\right)  $ is a Fr\'{e}chet space and for any linear partial differential
operator $L$ of degree $r,$ we have $\left\Vert L.f\right\Vert _{n}%
\leq\left\Vert f\right\Vert _{n+r}$; so $L$ is a shift operator of base $n$
and type $\left(  r,0\right)  $ (tame operator in Hamilton's terminology).
\end{example}

\subsection{Projective limit of shift operators\label{__ProjectiveLimitOfShiftOperators}}

\begin{lemma}
\label{L_BanachSpaceLrssn}\bigskip For any integer $n\geq s$, the following
set%
\[%
\begin{array}
[c]{cc}%
\mathcal{L}_{s,n}^{r,s}\left(  \mathbb{F}_{1},\mathbb{F}_{2}\right)  = &
\left\{
\begin{array}
[c]{c}%
\left(  L_{s},\dots,L_{n}\right)  \in\mathcal{L}_{s}^{r,s}\left(
\mathbb{F}_{1},\mathbb{F}_{2}\right)  \times\cdots\times\mathcal{L}_{n}%
^{r,s}\left(  \mathbb{F}_{1},\mathbb{F}_{2}\right)  :\\
\forall\left(  i,j\right)  \in\mathbb{N}^{2}:n\geq j\geq i\geq
s,{\vphantom{A}}_{2}\delta_{i-s}^{j-s}\circ L_{j}=L_{i}\circ{\vphantom{A}}%
_{1}\delta_{i+r}^{j+r}%
\end{array}
\right\}
\end{array}
\]
can be endowed with a structure of Banach space relatively to the norm
$\left\Vert \ \right\Vert _{s,n}^{r,s}$ defined by%
\[
\left\Vert \left(  L_{s},\dots,L_{n}\right)  \right\Vert _{s,n}^{r,s}%
=\sum\limits_{i=s}^{n}\left\Vert L_{i}\right\Vert _{L_{i}^{r,s}}%
\]

\end{lemma}

\begin{proof}
Since $\mathcal{L}_{s,n}^{r,s}\left(  \mathbb{F}_{1},\mathbb{F}_{2}\right)  $
is a closed subspace of the Banach space $\mathcal{L}_{s}^{r,s}\left(
\mathbb{F}_{1},\mathbb{F}_{2}\right)  \times\cdots\times\mathcal{L}_{n}%
^{r,s}\left(  \mathbb{F}_{1},\mathbb{F}_{2}\right)  $, it is also a Banach space.
\end{proof}

\begin{lemma}
\label{L_CanonicalProjectionLrsjsLrsis}For $j\geq i\geq s$, the canonical
projections
\[%
\begin{array}
[c]{cccc}%
\pi_{i}^{j}: & \mathcal{L}_{s,j}^{r,s}\left(  \mathbb{F}_{1},\mathbb{F}%
_{2}\right)  & \longrightarrow & \mathcal{L}_{s,i}^{r,s}\left(  \mathbb{F}%
_{1},\mathbb{F}_{2}\right) \\
& \left(  L_{s},\dots,L_{j}\right)  & \longmapsto & \left(  L_{s},\dots
,L_{i}\right)
\end{array}
\]
are linear and continuous.
\end{lemma}

\begin{proof}
For $j\geq i\geq s$, the linearity of $\pi_{i}^{j}$ is obvious.\newline The
continuity of $\pi_{i}^{j}$ is a consequence of%
\[%
\begin{array}
[c]{ll}%
\left\Vert \pi_{i}^{j}\left(  L_{s},\dots,L_{j}\right)  \right\Vert
_{s,i}^{r,s} & =\left\Vert \left(  L_{s},\dots,L_{i}\right)  \right\Vert
_{s,i}^{r,s}\\
& =\sum\limits_{k=s}^{i}\left\Vert L_{k}\right\Vert _{L_{k}^{r,s}}\\
& \leq\sum\limits_{k=s}^{j}\left\Vert L_{k}\right\Vert _{L_{k}^{r,s}}\\
& =\left\Vert \left(  L_{s},\dots,L_{j}\right)  \right\Vert _{s,j}^{r,s}%
\end{array}
\]

\end{proof}

We then have the following result.

\begin{theorem}
\label{T_FrechetStructureOnProjectiveLimitOfLnrs}$\left\{  \left(
\mathcal{L}_{s,i}^{r,s}\left(  \mathbb{F}_{1},\mathbb{F}_{2}\right)  ,\pi
_{i}^{j}\right)  \right\}  _{\left(  i,j\right)  \in\mathbb{N}^{2},\ j\geq
i\geq s}$is a projective sequence of Banach spaces whose projective limit
$\mathcal{L}^{r,s}\left(  \mathbb{F}_{1},\mathbb{F}_{2}\right)  $ can be
endowed with a Fr\'{e}chet structure.
\end{theorem}

\begin{proof}
For $k\geq j\geq i\geq s$, it is obvious that $\pi_{i}^{k}=\pi_{i}^{j}\circ
\pi_{j}^{k}$. Thus, according to Lemma \ref{L_BanachSpaceLrssn} and Lemma
\ref{L_CanonicalProjectionLrsjsLrsis}, $\left\{  \left(  \mathcal{L}%
_{s,i}^{r,s}\left(  \mathbb{F}_{1},\mathbb{F}_{2}\right)  ,\pi_{i}^{j}\right)
\right\}  _{\left(  i,j\right)  \in\mathbb{N}^{2},\ j\geq i\geq s}$ is a
projective sequence of Banach spaces. So its projective limit can be endowed
with a structure of Fr\'{e}chet space (cf. \ref{__FrechetSpaces}).
\end{proof}

\subsection{Inductive dual\label{__InductiveDual}}

Because the dual of a Fr\'{e}chet space generally drops out of the Fr\'{e}chet
category, it will be replaced by the inductive dual which is defined as a
projective limit of Banach spaces.

Let $\mathbb{F}$ be a graded Fr\'{e}chet space and let $\left(  \mathbb{F}%
_{n}\right)  _{n\in\mathbb{N}}$ be the sequence of associated Banach spaces.
We then consider, for $n\in\mathbb{N}$, the following space%
\[
\mathbb{F}_{n}^{0}=\left\{  \widehat{\omega_{n}}=\left(  \omega_{0}%
,\dots\omega_{n}\right)  \in\prod\limits_{i=0}^{n}\mathbb{F}_{i}^{^{\prime}%
}\right\}
\]

where $\mathbb{F}_{i}^{^{\prime}}$ is the topological dual of the Banach space
$\mathbb{F}_{i}$. Then $\mathbb{F}_{n}^{0}$ is a Banach space for the norm
$\left\Vert \ \right\Vert ^{n}$ defined by%
\[
\left\Vert \widehat{\omega_{n}}\right\Vert ^{n}=\sum\limits_{i=0}^{n}%
\max_{\left\Vert x_{i}\right\Vert _{i}=1}\left\vert \omega_{i}\left(
x_{i}\right)  \right\vert
\]

\begin{definition}
The projectif limit of the sequence $\left\{  \left(  \mathbb{F}_{n}^{0}%
,\Pi_{n}^{n+1}\right)  \right\}  _{n\in\mathbb{N}^{\ast}}$, where $\Pi
_{n}^{n+1}:\mathbb{F}_{n+1}^{0}\longrightarrow\mathbb{F}_{n}^{0}$ is the
natural projection, is called the inductive dual of $\mathbb{F}$ et denoted by
$\mathbb{F}^{0}$.
\end{definition}

The inductive dual $\mathbb{F}^{0}$ is a graded Fr\'{e}chet space.

The \textit{inductive cotangent bundle} $T^{0}\mathbb{F}$ is defined as the
trivial bundle of base $\mathbb{F}$ and fiber $\mathbb{F}^{0}$ and appears as
as the projective limit of $\left(  \mathbb{F}_{n}\times\mathbb{F}_{n}%
^{0},\left\Vert \ \right\Vert _{n}+\left\Vert \ \right\Vert ^{n}\right)  $. An
\textit{inductive differential form} is a smooth section of this bundle.

\section{Projective sequence of local shift morphisms\label{_ProjectiveSequenceOfLocalShiftMorphisms}}

\subsection{Local shift morphisms\label{__LocalShiftMorphisms}}

Let $\mathbb{F}_{1}$ (resp. $\mathbb{F}_{2},\mathbb{F}_{3}$) be a graded
Fr\'{e}chet space and let $\left(  \mathbb{F}_{1}^{n},\left\Vert \ \right\Vert
_{n}^{1}\right)  _{n\in\mathbb{N}}$ (resp. $\left(  \mathbb{F}_{2}%
^{n},\left\Vert \ \right\Vert _{n}^{2}\right)  _{n\in\mathbb{N}},\left(
\mathbb{F}_{3}^{n},\left\Vert \ \right\Vert _{n}^{3}\right)  _{n\in\mathbb{N}%
}$) be the sequence of associated local Banach spaces.

\begin{definition}
\label{D_LocalShiftMorphismType_rs} Let $n\in\mathbb{N}$ such that $n-s\geq0$.
A smooth map%
\[
\varphi:U_{n}\longrightarrow\mathcal{L}\left(  \mathbb{F}_{2}^{n+r}%
,\mathbb{F}_{3}^{n-s}\right)
\]
where $U_{n}$ is an open set of $\mathbb{F}_{1}^{n}$, is called a local shift
morphism of base $n$ and type $\left(  r,s\right)  \in\mathbb{N}%
\times\mathbb{N}$ above $U_{n}$.
\end{definition}

\subsection{Projective sequence of local shift morphisms\label{__ProjectiveSequenceOfLocalShiftMorphisms}}

\begin{definition}
\label{D_ProjectiveSequenceOfLocalShiftMorphisms} A sequence $\left(
\varphi_{n}\right)  _{n\in\mathbb{N},\ n\geq s}$ of local shift morphisms
$\varphi_{n}$ of type $\left(  r,s\right)  \in\mathbb{N}\times\mathbb{N}$
above $U_{n}$ is said to be a projective sequence of local shift morphisms if

\begin{description}
\item[\textbf{(PSLSM 1)}] $U_{s}\supset U_{s+1}\supset\cdots\supset
U_{n}\supset U_{n+1}\supset\cdots$ and $U=\bigcap\limits_{n=s}^{+\infty}U_{n}$
is a non empty open set of $\mathbb{F}_{1}$;

\item[\textbf{(PSLSM 2)}] For any $q=(q_{n})_{n\in\mathbb{N}}\in U$, we have
the following commutative diagram:
\end{description}
$
\xymatrix {
	&& U_n \times \mathbb{F}_2^{n+r} \ar[dll]_{(\operatorname{Id}_{U_n},\,\varphi_n(q_n))\;\;\;} \ar[ddl] &&&
	U_{n+1} \times \mathbb{F}_2^{n+r+1} \ar[lll]_{{\vphantom{A}}_{1}\delta_n^{n+1} \times {\vphantom{A}}_{2}\delta_{n+r}^{n+r+1}} \ar[dll]_{(\operatorname{Id}_{U_{n+1}},\,\varphi_{n+1}(q_{n+1}))\;\;\;} \ar[ddl]^{{\vphantom{A}}_{2}\pi_{n+1}^{n+1+r}} \\
	U_n  \times \mathbb{F}_3^{n-s}\ar[dr]_{{\vphantom{A}}_{3}\pi_n^{n-s}} &&& U_{n+1}  \times \mathbb{F}_3^{n-s+1} \ar[lll]_{\;\;{\vphantom{A}}_{1}\delta_n^{n+1} \times {\vphantom{A}}_{3}\delta_{n-s}^{n-s+1}} \ar[dr] && \\
	& U_n &&& U_{n+1} \ar[lll]_{{\vphantom{A}}_{1}\delta_n^{n+1}} &\\
}
$

\end{definition}

\begin{theorem}
\label{T_SmoothStructureOnProjectiveLimitOfLocalShiftMorphims}The projective
limit $\underleftarrow{\lim}\varphi_{n}$ of a projective sequence of local
shift morphisms $\varphi_{n}$ of type $\left(  r,s\right)  \in\mathbb{N}%
\times\mathbb{N}$ above $U_{n}$ is a smooth map from the open set
$U=\bigcap\limits_{n=s}^{+\infty}U_{n}$ of the Fr\'{e}chet space
$\mathbb{F}_{1}$ to the Fr\'{e}chet space $\mathcal{L}^{r,s}\left(
\mathbb{F}_{2},\mathbb{F}_{3}\right)  $.
\end{theorem}

\begin{proof}
Since $\mathcal{L}^{r,s}\left(  \mathbb{F}_{2},\mathbb{F}_{3}\right)  $ is the
projective limit of the Banach spaces $\mathcal{L}_{n}^{r,s}\left(
\mathbb{F}_{2},\mathbb{F}_{3}\right)  $ (cf. Theorem
\ref{T_FrechetStructureOnProjectiveLimitOfLnrs}) the smoothness of
$\underleftarrow{\lim}\varphi_{n}$ results from the smoothness of the maps
$\varphi_{n}$ and the Proposition
\ref{P_DiffentiabilityOfProjectiveLimitsOfDifferentialFunctions}.
\end{proof}

\section{Hilbert towers \label{_HilbertTowers}}

In this section, the reader is referred to \cite{KapMak}. \newline We consider
the particular case where the Fr\'{e}chet spaces $\mathbb{F}_{1}$,
$\mathbb{F}_{2}$ and $\mathbb{F}_{3}$ are all equal to a same projective limit
of Hilbert spaces.

\subsection{Definition. Example}

\begin{definition}
\label{D_HilbertTower} The sequence $\left(  H_{n}\right)  _{n\in\mathbb{N}}$
is a Hilbert tower if

\begin{description}
\item[\textbf{(HT 1)}] $\left(  H_{n}\right)  _{n\in\mathbb{N}}$ is a
decreasing sequence of Hilbert spaces: $H_{0}\supset H_{1}\supset\cdots$;

\item[\textbf{(HT 2)}] $\forall n\in\mathbb{N},\ \overline{H_{n+1}}=H_{n}$;

\item[\textbf{(HT 3)}] There exists a basis of $H_{\infty}=\bigcap
\limits_{n\in\mathbb{N}}H_{n}$, i.e. an orthonormal basis $\left(
e_{m}\right)  _{m\in\mathbb{N}}$ of $H_{0}$, where $e_{m}\in H_{\infty}$, such
that $\left(  e_{m}\right)  _{m\in\mathbb{N}}$ is a basis of any $H_{N}$ (with
$N\in\mathbb{N}$).
\end{description}
\end{definition}

A Hilbert tower can be seen as an IHL space as defined in \cite{Omo}.

\begin{example}
\label{Ex_HilbertTowerSobolevSpaces} The sequence of Sobolev spaces $\left(
H^{n}\left(  \mathbb{S}^{1}\right)  \right)  _{n\in\mathbb{N}}$ where
\[
H^{n}\left(  \mathbb{S}^{1}\right)  =\left\{  q\in L^{2}\left(  \mathbb{S}%
^{1}\right)  :\forall k\in\left\{  0,\dots,n\right\}  ,q^{\left(  k\right)
}\in L^{2}\left(  \mathbb{S}^{1}\right)  \right\}
\]
is a Hilbert tower where the orthonormal basis is $\left(  e_{0},e_{1}%
,e_{-1},\dots,e_{k},e_{-k},\dots\right)  $, ($k\in\mathbb{N}$) where
$e_{k}:x\mapsto e^{i2k\pi x}$.
\end{example}

Let $\left(  H_{n}\right)  _{n\in\mathbb{N}}$ be a Hilbert tower where
$\iota_{n}^{n+1}:H_{n+1}\longrightarrow H_{n}$ is the natural injection and
let us denote $\left\langle .,.\right\rangle _{n}$ the inner product of
$H_{n}$ and $\left\Vert \ \right\Vert _{H_{n}}$ the associated norm.

The projective limit $H_{\infty}$ of the Hilbert tower $\left(  H_{n}\right)
_{n\in\mathbb{N}}$ is perfectly defined and can be endowed with a structure of
Fr\'{e}chet space.

\subsection{Local shift Hilbert morphisms\label{__LocalShiftHilbertMorphisms}}

In the sequel, we reformulate some of the precedent results in the particular
case of a Hilbert tower $\left(  H_{n}\right)  _{n\in\mathbb{N}}$, that is for
all $n\in\mathbb{N},\mathbb{F}_{1}^{n}=\mathbb{F}_{2}^{n}=\mathbb{F}_{2}%
^{n}=H_{n}$, where the norm $\left\Vert \ \right\Vert _{1}^{n}=\left\Vert
\ \right\Vert _{2}^{n}=\left\Vert \ \right\Vert _{3}^{n}=\sqrt{\left\langle
.,.\right\rangle _{n}}$ are associated to the inner product of $H_{n}$.

\begin{definition}
\label{D_LocalShiftMorphism}A local shift Hilbert morphism of base $n$ and
type $\left(  r,s\right)  $ is a smooth map%
\[
\varphi_{n}:U_{n}\longrightarrow\mathcal{L}\left(  H_{n+r},H_{n-s}\right)
\]
where $U_{n}$ is an open set of $H_{n}$.
\end{definition}

\begin{example}
\label{Ex_FirstPoissonStructureForKdVEquation} On the Sobolev tower $\left(
H_{n}=H^{n}\left(  \mathbb{S}^{1}\right)  \right)  _{n\in\mathbb{N}}$ (cf.
Example \ref{Ex_HilbertTowerSobolevSpaces}), we consider the operator
\[%
\begin{array}
[c]{cccc}%
\partial_{x}: & U\cap H_{n} & \longrightarrow & \mathcal{L}\left(
H_{n+1},H_{n}\right)  \\
& q & \longmapsto & \left(  \partial_{x}\right)  _{q}%
\end{array}
\]
which corresponds to the first Poisson structure for the KdV equation (cf.
section \ref{_ExampleKdVEquationS1}.) where $U=H_{0}=H^{0}\left(
\mathbb{S}^{1}\right)  $ and
\[%
\begin{array}
[c]{cccc}%
\left(  \partial_{x}\right)  _{q}: & H_{n+1} & \longrightarrow & H_{n}\\
& u & \longmapsto & \partial_{x}u
\end{array}
.
\]

So $\partial_{x}$ is a local shift Hilbert morphism of type $\left(
1,0\right)  $ above any $H_{n}=H^{n}\left(  \mathbb{S}^{1}\right)  $.
\end{example}

\begin{example}
\label{Ex_SecondPoissonStructureForKdVEquation}On the Sobolev tower $\left(
H^{n}\left(  \mathbb{S}^{1}\right)  \right)  _{n\in\mathbb{N}}$, the operator
\[%
\begin{array}
[c]{cccc}%
L_{n}: & U\cap H_{n} & \longrightarrow & \mathcal{L}\left(  H_{n+2}%
,H_{n-1}\right)  \\
& q & \longmapsto & \left(  L_{n}\right)  _{q}%
\end{array}
\]
corresponds to the second Poisson structure for the KdV equation where
$U=H_{0}=H^{0}\left(  \mathbb{S}^{1}\right)  $ and%
\[%
\begin{array}
[c]{cccc}%
\left(  L_{n}\right)  _{q}: & H_{n+2} & \longrightarrow & H_{n-1}\\
& u & \longmapsto & -\dfrac{1}{2}\partial_{x}^{3}u+q.\partial_{x}%
u+\partial_{x}q.u
\end{array}
.
\]

$L_{n}$ is then a local shift morphism of type $\left(  2,1\right)  $ above
$H_{n}=H^{n}\left(  \mathbb{S}^{1}\right)  $.

In particular, we have, for $q\in H^{n}\left(  \mathbb{S}^{1}\right)  $,
\[
\left(  L_{n}\right)  _{q}\in\mathcal{L}\left(  H^{n+2}\left(  \mathbb{S}%
^{1}\right)  ,H^{n-1}\left(  \mathbb{S}^{1}\right)  \right)
\]
because
\[
\forall u\in H^{n+2}\left(  \mathbb{S}^{1}\right)  ,\ \left\Vert \left(
L_{n}\right)  _{q}\left(  u\right)  \right\Vert _{n-1}\leq c_{n}\left\Vert
u\right\Vert _{n+2}%
\]
where the norm $\left\Vert \ \right\Vert _{n}$ is given by%
\[
\left\Vert v\right\Vert _{n}=\sqrt{\sum\limits_{k=0}^{n}\int
\nolimits_{\mathbb{S}^{1}}\left[  \left(  \partial_{x}^{k}v\right)  \left(
x\right)  \right]  ^{2}dx}\text{.}%
\]

\end{example}

\subsection{Projective limits of local shift Hilbert morphisms\label{__ProjectiveLimitLocalShiftHilbertMorphisms}}

\begin{definition}
\label{D_ProjectiveSequenceOfLocalShiftHilbertMorphisms} Let $\left(
H_{n}\right)  _{n\in\mathbb{N}}$ be a Hilbert tower. A sequence $\left(
\varphi_{n}\right)  _{n\in\mathbb{N},\ n\geq s}$ of local shift morphisms
$\varphi_{n}$ of type $\left(  r,s\right)  \in\mathbb{N}\times\mathbb{N}$
above $H_{n}$ is said to be a projective sequence of local shift Hilbert
morphisms if, for any $q=(q_{n})\in\prod\limits_{n\in N}H_{n}$, we have the
following commutative diagram:
\[
\xymatrix {
	&& H_n \times H_{n+r} \ar[dll]_{(\operatorname{Id}_{H_n},\,\varphi_n(q_n))\;\;\;} \ar[ddl] &&&
	H_{n+1} \times H_{n+r+1} \ar[lll]_{\iota_n^{n+1} \times \iota_{n+r}^{n+r+1}} \ar[dll]_{(\operatorname{Id}_{H_{n+1}},\,\varphi_n(q_{n+1}))\;\;\;} \ar[ddl]^{\pi_{n+1}^{n+1+r}} \\
	H_n  \times H_{n-s}\ar[dr]_{\pi_n^{n-s}} &&& H_{n+1} \times H_{n-s+1} \ar[lll]_{\;\;\iota_n^{n+1} \times \iota_{n-s}^{n-s+1}} \ar[dr] && \\
	& H_n &&& H_{n+1} \ar[lll]_{\iota_n^{n+1}} &\\
}
\]

\end{definition}

Let $\left(  H_{n}\right)  _{n\in\mathbb{N}}$ be a Hilbert tower and consider
$H_{\infty}=\bigcap\limits_{n\in\mathbb{N}}H_{n}=\underleftarrow{\lim}H_{n}$.
For $n\geq s$, the space%
\[%
\begin{array}
[c]{cc}%
\mathcal{H}_{s,n}^{r,s}\left(  H_{\infty}\right)  = & \left\{
\begin{array}
[c]{c}%
\left(  L_{s},\dots,L_{n}\right)  \in\prod\limits_{i=s}^{n}\mathcal{L}\left(
H_{i+r},H_{i-s}\right)  :\\
\forall\left(  i,j\right)  \in\mathbb{N}^{2}:n\geq j\geq i\geq s,\iota
_{i-s}^{j-s}\circ L_{j}=L_{i}\circ{\iota}_{i+r}^{j+r}%
\end{array}
\right\}
\end{array}
\]
is a Banach space. We then get a projective sequence $\left\{  \left(
\mathcal{H}_{s,i}^{r,s}\left(  H_{\infty}\right)  ,\pi_{i}^{j}\right)
\right\}  _{\left(  i,j\right)  \in\mathbb{N}^{2},\ j\geq i\geq s}$ where
\[
\pi_{i}^{j}:\left(  L_{s},\dots,L_{j}\right)  \mapsto\left(  L_{s},\dots
,L_{i}\right)  .
\]
Its projective limit $\mathcal{H}^{r,s}\left(  H_{\infty}\right)  $ can be
endowed with a structure of Fr\'{e}chet space

For a projective sequence of local shift Hilbert morphisms $\left(
\varphi_{n}\right)  _{n\in\mathbb{N},\ n\geq s}$ of type $\left(  r,s\right)
$, we have the following commutative diagram:%
\[%
\begin{array}
[c]{ccc}%
\mathcal{H}_{s,i}^{r,s}\left(  H_{\infty}\right)  & \overset{\pi_{i}^{j}%
}{\longleftarrow} & \mathcal{H}_{s,j}^{r,s}\left(  H_{\infty}\right) \\
\left(  \varphi_{s},\dots,\varphi_{i}\right)  \uparrow
\ \ \ \ \ \ \ \ \ \ \ \ \ \ \ \ \  &  &
\ \ \ \ \ \ \ \ \ \ \ \ \ \ \ \ \uparrow\left(  \varphi_{s},\cdots,\varphi
_{j}\right) \\
U\cap H_{s}\times\cdots\times U\cap H_{i} & \overset{p_{i}^{j}}{\longleftarrow
} & U\cap H_{s}\times\cdots\times U\cap H_{j}%
\end{array}
\]
where the maps $\left(  \varphi_{s},\dots,\varphi_{n}\right)  :U\cap
H_{s}\times\cdots\times U\cap H_{n}\longrightarrow\mathcal{H}_{s,n}%
^{r,-s}\left(  H_{\infty}\right)  $ are smooth.

We can define the projective limit
\[
\varphi=\underleftarrow{\lim}\left(  \varphi_{s},\dots,\varphi_{n}\right)
:U\cap H_{\infty}\longrightarrow\mathcal{H}^{r,s}\left(  H_{\infty}\right)
\]
and this limit is smooth.

\begin{example}
The sequence $\left(  L_{n}\right)  _{n\in\mathbb{N}}$ of Example
\ref{Ex_SecondPoissonStructureForKdVEquation} is a projective sequence of
local shift morphisms of type $\left(  2,1\right)  $.
\end{example}

\section{Shift Hilbert Poisson tensors \label{_ShiftHilbertPoissonTensors}}

The notion of Poisson tensor is relevant in Mechanics and Mathematical
Physics. It corresponds to a tensor field $P$ twice contravariant whose
Schouten bracket $\left[  P,P\right]  $ vanishes. Bihamiltonian structures
corresponding to a pair of compatible Poisson tensors is a fundamental tool in
the resolution of some dynamical systems because the recursion operator
linking both structures gives rise to conservation laws.

In the framework of Hilbert towers, thanks to the identification of a Hilbert
space with its dual (Riesz Theorem), the morphism $P$ from the cotangent
bundle to the tangent bundle can be seen as a projective limit of local shift
Hilbert morphisms. Such objects are adapted to the description of the KdV
equation on the circle $\mathbb{S}^{1}$.

\begin{definition}
\label{D_ShiftHilbertPoissonTensor}Let $\left(  P_{n}\right)  _{n\in
\mathbb{N}}$ be a sequence of local shift morphisms of type $\left(
r,s\right)  $ on the Hilbert tower $\left(  H_{n}\right)  _{n\in\mathbb{N}}$
whose projective limit is $P=\underleftarrow{\lim}P_{n}$. \newline$P$ is said
to be a shift Hilbert Poisson tensor of type $\left(  r,s\right)  $ on
$H_{\infty}=\underleftarrow{\lim}H_{n}$ if, for any $q=\underleftarrow{\lim
}q_{n}$, $f=\underleftarrow{\lim}f_{n}$, $g=\underleftarrow{\lim}g_{n}$ and
$h=\underleftarrow{\lim}h_{n}$, it fulfils the following conditions:

\begin{description}
\item[\textbf{(SHPT 1)}] $P$ is antisymmetric,\newline i.e. for all
$n\in\mathbb{N}$ such that $n-s\geq0$,
\[
\left\langle \left(  P_{n}\right)  _{q_{n}}\left(  f_{n+r}\right)
,g_{n-s}\right\rangle _{H_{n-s}}=-\left\langle \left(  P_{n}\right)  _{q_{n}%
}\left(  g_{n+r}\right)  ,f_{n-s}\right\rangle _{H_{n-s}}%
\]

\item[\textbf{(SHPT 2)}] The Schouten bracket vanishes: $\left[  P,P\right]
=0$,\newline where for all $n\in\mathbb{N}$ such that $n+r-2s\geq0$,
\[
\ \left[  P_{n},P_{n}\right]  _{q_{n}}\left(  f_{n+r},g_{n+r},h_{n+r}\right)
=\mathbf{\sigma}\left\langle f_{n+r-2s},P_{q_{n-s}}^{\prime}\left(
g_{n+r-s},\left(  P_{n}\right)  _{q_{n}}h_{n+r}\right)  \right\rangle \
\]

\end{description}
\end{definition}

In this definition, the differentiabity of $P$ at $q$ is given by:%
\[
P_{q}^{\prime}\left(  f,g\right)  =\dfrac{d}{dt}P_{q+tg}f_{\ |t=0}%
\]

\begin{example}
\label{Ex_KdVEquationS1} The Korteweg-de Vries (KdV) equation (\cite{KorVri})
is an evolution equation in one space dimension which was proposed as a model
to describe waves on shallow water surfaces. This nonlinear and dispersive PDE
was first introduced by J. Boussinesq (\cite{Bous}) and rediscovered by D.
Korteweg and G. de Vries (\cite{KorVri}) in order to modelize natural
phenomena discovered by Russel (\cite{Rus}).\newline\newline In \cite{Arn}, V.
Arnold suggested a general framework for the Euler equations on an arbitrary
group that describe a geodesic flow with respect to a suitable one-sided
invariant Riemannian metric on the group. This approach works for the Virasoro
group and provides a natural geometric setting for the KdV equation (cf.
\cite{KheMis}).

It is well known (e.g. \cite{FMPZ}, \cite{MagMor}, \cite{Olv}, \cite{Sch}, \cite{ZubMag}, ...) that this equation can be written in Hamiltonian form in two distinct
ways. Moreover, there exists an infinite hierarchy of commuting conservation
laws and Hamiltonian flows generated by a recursion operator linking both
Poisson brackets. Such an equation can be viewed as a complete integrable
system and has a lot of remarkable properties, including soliton solutions. 
In \cite{KisLeu}, the framework of variational Lie algebroids is used to describe such an evolutionary equation.

Here we consider the KdV equation on the circle $\mathbb{S}^{1}$ of unit
length
\[
\partial_{t}u=-\partial_{x}^{3}u+6u\partial_{x}u
\]
where $t\in\mathbb{R}$ and $x\in\mathbb{S}^{1}.$\newline

This equation can be seen as an infinite dimensional system on the Hilbert
tower $\left(  H^{n}\left(  \mathbb{S}^{1}\right)  \right)  _{n\in\mathbb{N}}$
(cf. \cite{KapMak} and \cite{KapPos}). This system can be written in a
bihamiltonian way relatively to the compatible shift Hilbert Poisson tensors
$\partial_{x}$, of type $\left(  1,0\right)  $, and $L_{q}$ of type $\left(
2,1\right)  $.
\end{example}

\noindent
Patrick Cabau\\
Lyc\'{e}e Pierre de Fermat\\
Parvis des Jacobins\\
BP 7013\\
31068 Toulouse Cedex 7\\
France\\
\noindent
e-mail: patrickcabau@gmail.com 

\end{document}